\documentclass{amsart}
\usepackage{xcolor}

\usepackage{amsmath,amssymb,amsthm,amsfonts,epsfig,color}

\usepackage{verbatim}
\usepackage{tikz-cd}

\theoremstyle{definition}
\newtheorem{thm}{Theorem}[section]
\newtheorem{prop}[thm]{Proposition}
\newtheorem{cor}[thm]{Corollary}
\newtheorem{lem}[thm]{Lemma}
\newtheorem{ex}[thm]{Example}

\newtheorem{rem}[thm]{Remark}

\newcommand{\h}{\mathcal{\mathcal H}}
\newcommand{\hsp}{\mathcal{H}}
\newcommand{\oL }{\mathcal{L}(E) }

\def\ni{\noindent}
\def\LL{{\mathcal L}}
\def\nn{\nonumber}
\def\ben{\begin{eqnarray}}
\def\eeqn{\end{eqnarray}}

\title[Matrix-valued truncated Toeplitz operators]{Matrix-valued truncated Toeplitz operators: Properties and Applications}
\author{Rewayat Khan}
\author{Flavia Colonna}

\address{Deaprtment of Artificial Intelligence and Data Science ,Istanbul Aydin University, Turkey}
\address{Department of Mathematical Sciences, George Mason University, Fairfax, VA, USA}
\email{rewayat.khan@gmail.com}
\email{fcolonna@gmu.edu}

\begin{document}
%
\maketitle
%

\begin{abstract}
In this paper, we present an application of matrix-valued truncated
Toeplitz operators to the matrix version of the Nevanlinna–Pick interpolation problem. We then investigate the unitary part of a matrix-valued truncated Toeplitz operator induced by a contractive matrix-valued symbol. 
We further prove that, for a pure matrix-valued inner function, the corresponding model space is invariant under multiplication operator if and only if the symbol is constant.
 In addition, we give an algebraic proof of a characterization of  the invertibility of matrix-valued asymmetric truncated Toeplitz operators using Toeplitz operators with $2\times 2$ block matrix symbol. We also discuss the solvability of matrix-valued asymmetric truncated Toeplitz operators. 
\end{abstract}
\medskip

\section{\bf{Introduction}}\label{intro}
 Let $\mathbb{C}$ be the complex plane and $\mathbb{T}$ the unit circle in $\mathbb{C}$. Let $L^{2}(\mathbb{T})$ be the space of all measurable functions on the unit circle $\mathbb{T}$ whose Fourier coefficients are square summable.
Let $H^2$ be the classical Hardy space on the unit disk 
$$\mathbb{D}=\{z\in\mathbb{C}:|z|<1\}.$$ The space $H^{2}$ can be thought of as a closed subspace of $L^{2}(\mathbb{T})$, with respect to the normalized Lebesgue measure $m$ on $\mathbb{T}$, consisting of functions whose negative Fourier coefficients vanish. 
 The space of essentially bounded functions in $L^2({\mathbb{T}})$ is denoted by $L^{\infty}$, while the space of bounded analytic functions in $H^2$ is denoted by $H^{\infty}$.

An inner function is a bounded analytic function $\theta\in H^{2}$ such that $|\theta(z)|=1$ a.e. on $\mathbb{T}$. A finite Blaschke product is an inner function of the form
$$b(z)=\prod_{j=1}^{n}\frac{\lambda_{j}-z}{1-\overline{\lambda}_{j}z}$$
where $\lambda_{j}\in \mathbb{D}$.

Truncated Toeplitz operators (TTO's) are compressions of multiplication operators to the backward shift invariant subspaces of $H^2$. Each of these subspaces is of the form
$$K_{\theta}=(\theta H^{2})^{\perp}=H^2\ominus \theta H^2,$$ where $\theta$ is an inner function.  In  \cite{sar}, Sarason gave origin to a systematic study of TTO's. Thanks to his work, a whole and rich direction of research developed. A model space $K_{\theta}$ is finite dimensional if and only if $\theta$ is a finite Blaschke product \cite[Proposition 5.7.6]{GMR}. 


Let $\theta_{1}$ and $\theta_{2}$ be two non-constant inner functions and let $K_{\theta_{1}}$ and $K_{\theta_{2}}$ be the respective corresponding model spaces. Let $\varphi\in L^{2}(\mathbb T)$ and denote by $P_{\theta_{2}}$ 
 the orthogonal projection onto $K_{\theta_{2}}$. An asymmetric truncated Toeplitz operator is the densely defined linear operator $A_\varphi^{\theta_1,\theta_2}$ 
 given by
\begin{equation}\label{11}
A_{\varphi}^{\theta_1, \theta_2}f=P_{\theta_2}(\varphi f),\quad f\in K_{\theta_{1}}\cap H^{\infty}.
\end{equation}
Let $\mathcal{T}(\theta_{1},\theta_{2})$ denote the space of all asymmetric truncated Toeplitz operators which can be extended as bounded operators from $K_{\theta_1}$ into $K_{\theta_2}$ (see \cite{part3, jl}).

\medskip

 For a complex separable Hilbert space $E$, let  $\mathcal{L}(E)$ denote the algebra of all bounded linear operators on $E$. We denote the norm and inner product in $E$ by $\|.\|_{E}$ and $\langle \cdot, \cdot \rangle_{E}$, respectively. Let $L^{2}(E)$ represent the Hilbert space of $E$-valued measurable functions that are square integrable on $\mathbb{T}$, that is, the collection of elements $f:\mathbb{T}\to E$ of the form
\begin{equation}\label{one1}
 f(z)=\sum\limits_{n=-\infty}^{\infty}a_{n}z^n\end{equation} 
 where $\{a_{n}\}\subset E$ and $\sum\limits_{n=-\infty}^{\infty}\|a_{n}\|_{E}^{2}<\infty.$
It can be readily verified that $L^{2}(E)$ is a separable Hilbert space with inner product given by
\begin{equation*}
\langle f,g\rangle_{L^{2}(E)}=\int_{\mathbb{T}}\langle f(z),g(z)\rangle_{E}\,dm(z),\quad f,g\in L^2(E).
\end{equation*}
If $f\in L^2(E)$ is as in 
 \eqref{one1}, then its Fourier series converges in the $L^2(E)$-norm and
$$\|f\|_{L^2(E)}^2=\int_{\mathbb{T}} \| f(z)\|_{E}^2\,dm(z)=\sum\limits_{n=-\infty}^{\infty}\|a_{n}\|_{E}^{2}.$$
The vector-valued Hardy-Hilbert space 
is the 
 subspace of $L^2(E)$ defined as
$$H^{2}(E)={\{f\in L^{2}(E): a_{n}=0 \text{ for all }  n\leq -1}\}.$$ 
As special cases, we obtain $L^{2}(\mathbb{C}^{n})$ and $H^{2}(\mathbb{C}^{n})$ when $E=\mathbb{C}^{n}$. 
 Similarly, we may 
  define the $\mathcal{L}(E)$-valued, i.e., operator-valued functions. The corresponding spaces are denoted by $L^{2}(\mathcal{L}(E))$ and $H^{2}(\mathcal{L}(E))$. If $\text{dim\,} E=n<\infty$, then $\mathcal{L}(E)$ corresponds to the space of $n\times n$ matrices.

\medskip

The space $L^{2}(\LL(E))$ is defined as the set consisting of all functions $\Phi:\mathbb{T}\to \LL(E)$ of the form 
$$\Phi(z)=\sum\limits_{n=-\infty}^{\infty}A_{n}z^n \ \text{ such that }\ \sum\limits_{n=-\infty}^{\infty}\|A_{n}\|_{\LL(E)}^{2}<\infty,$$
while 
 the space $H^{2}(\LL(E))$ is the set of $\Phi\in L^{2}(\LL(E))$ with $A_{n}=0$ for all $n\leq -1$.
The space of essentially bounded functions in $L^{2}(\LL(E))$ is denoted by $L^{\infty}(\LL(E))$, while the space of bounded functions in $H^{2}(\LL(E))$ is denoted by $H^{\infty}(\LL(E))$.

\medskip
Let $\Phi\in L^{\infty}(\mathcal{L}(E))$. The multiplication operator $L_\Phi: L^2(E)\to L^2(E)$ is defined by
$$L_{\Phi}f=\Phi f, \quad ~~f\in L^{2}(E).$$
The  block Toeplitz operator $T_{\Phi}$ is defined by
$$T_{\Phi}f=P_{+}(L_\Phi f), \quad ~~f\in H^{2}(E),$$
where $P_{+}$ denotes the orthogonal projection from $L^{2}(E)$ onto $H^{2}(E)$. The operator $S=T_{zI_{E}}$ is an example of a block Toeplitz operator.

For each $\lambda \in \mathbb{D}$ and $x\in E$, the function $$k_\lambda x=\frac{1}{1-\overline{\lambda} z}I_{E}\,x,$$ belongs to $H^2(E)$ and is a reproducing kernel function for $H^2(E)$ since it has the reproducing property
$$\langle f, k_{\lambda}x\rangle=\langle f(\lambda), x\rangle.$$

A function $B(z)\in \LL(E)$ is called a finite
Blaschke-Potapov product if it admits a factorization of the form
\begin{equation}\label{BP}
B=UB_{1}B_{2} \cdots B_{n},
\end{equation}
where $U$ is a unitary operator and, for each $j=1,\dots,n$,
$$B_{j}(z)=\frac{\lambda_{j}-z}{1-\overline{\lambda}_{j}z}P_{j}+(I-P_{j})$$
for some $\lambda_{j}\in \mathbb{D}$ and $P_{j}$ on $E$ is an orthogonal projection. 
The degree of the Blaschke-Potapov product in \eqref{BP} is defined by
$$\text{deg\,} B
=\sum\limits_{j=1}^{n}\text{rank\,} P_{j}.$$
 Note that $B$ in \eqref{BP} is an inner function.

Let $\Theta$ be an operator-valued inner function, that is,  a function with values in $\oL $ (the algebra of all bounded linear operators on $E$), analytic in $\mathbb{D}$, bounded and such that the boundary values $\Theta(z)$ are unitary operators a.e. on $\mathbb{T}$. An inner function $\Theta$ is said to be pure  if $\|\Theta(0)\|<1$.  A vector valued model space $K_{\Theta}\subset H^{2}(E)$ is the orthogonal complement of $\Theta H^{2}(E)$, that is, $$K_{\Theta}= H^{2}(E)\ominus  \Theta H^{2}(E).$$     
This space appears in connection with model theory of Hilbert space contractions (see \cite{NF}).
 \medskip

In the theory of contractions on a Hilbert space, these model spaces are the scalar
case of a more general construction, which provides functional models for arbitrary
completely non-unitary contractions. In particular, it makes sense to study 
 matrix-valued inner functions $\Theta$ and the corresponding model space $K_{\Theta}$. 

\medskip
Completely non-unitary (c.n.u) contractions have played an important role in the classification of contractions. The most important result is the functional model for c.n.u. contractions established
by Nagy and Foias \cite{NF}.  It is evident from their result we state in Theorem~\ref{ucnu}, that the most natural approach for finding conditions for c.n.u. contractions relies on the understanding of the unitary subspace.
As an additional resource we refer the reader to \cite{BT2}.

\medskip

From the theory of truncated Toeplitz operators initiated by Sarason and its subsequent developments in the literature regarding generalizations to matrix versions, the challenging question that arises is how the operator theoretic properties of the matrix-valued operator $A_{\Phi}^{\Theta}$ defined below relate to the theoretic properties of $\Phi$.

\medskip

Let $\Phi\in L^2(\oL)$, and $\Theta_1,\Theta_2\in H^{\infty}(\oL)$ be two inner functions, let
\begin{equation}\label{11}
A_{\Phi}^{\Theta_{1}, \Theta_{2}}f=P_{\Theta_{2}}(\Phi f),\quad f\in K_{\Theta_{1}}\cap H^{\infty}(E).
\end{equation} 
Here $P_{\Theta_2}$ is the orthogonal projection from $L^2(E)$ onto $K_{\Theta_{2}}$.
The operator $A_{\Phi}^{\Theta_{1}, \Theta_{2}}$ is called a matrix-valued asymmetric truncated Toeplitz operator, while $A_{\Phi}^{\Theta_{1}}=A_{\Phi}^{\Theta_{1}, \Theta_{1}}$ is called a matrix-valued truncated Toeplitz operator (MTTO, see \cite{RT}). Both are densely defined on $K_{\Theta_{1}}$. Let $\mathcal{T}(\Theta_{1},\Theta_{2})$ be the set of all matrix-valued asymmetric truncated Toeplitz operators of the form \eqref{11} which can be extended boundedly to the whole space $K_{\Theta_{1}}$ and for $\Theta_{1}=\Theta_{2}=\Theta$ let $\mathcal{T}(\Theta )=\mathcal{T}(\Theta,\Theta)$. The basics of matrix-valued truncated Toeplitz operators were 
developed in \cite{RT}.

 Two important examples of operators from $\mathcal{T}(\Theta )$ are the model operators
$$S_{\Theta}=A_{z }^{\Theta }=A_{zI_{E}}^{\Theta }\quad\text{and}\quad S_{\Theta }^*=A_{\bar z }^{\Theta }=A_{\bar zI_{E}}^{\Theta}.$$

As shown in \cite{sar}, a number of new questions arise, mostly due to the non-commutative nature of matrices.

\medskip

\medskip

\vskip2pt

\ni {\bf Organization of the paper.} 
In Section~\ref{s2}, after giving some introductory materials,  we give an application of the matrix-valued truncated Toeplitz operators to prove in Theorem~\ref{NevPick} the matrix version of the Nevanlinna-Pick interpolation problem. 

In Section~\ref{s3}, we characterize the unitary and completely non-unitary parts of matrix-valued truncated Toeplitz operators. 
In Theorem~\ref{inv}, we prove the $S_\Theta$-invariance of the unitary space when the model space $K_\Theta$  is invariant under the multiplication operator $L_\Phi$ induced by a pure inner function $\Theta$. We then provide an example to show that the converse is false. In Theorem~\ref{Phi constant}, we show that the model space $K_\Theta$ is invariant under $L_\Phi$ if and only if the symbol $\Phi$ is constant. 

Section~\ref {s4} deals with the invertibility and the solvability of matrix-valued asymmetric truncated Toeplitz operators. Specifically, in Theorem~\ref{G}, we provide a short algebraic proof of a characterization of 
 the invertibility of matrix-valued asymmetric truncated Toeplitz operators using Toeplitz
operators with $2 \times 2$ block matrix symbol. Finally, in Theorem~\ref{ATTO}, we discuss the normal solvability of matrix-valued asymmetric truncated Toeplitz operators.

\section{\bf{Applications of matrix-valued} truncated Toeplitz operators}\label{s2}
In this section, we give an application of the matrix-valued truncated Toeplitz operators to prove the matrix version of Nevanlinna-Pick interpolation problem. 

The following lemma characterizes the finite dimensional vectorial model spaces.

 \begin{lem}\cite[Lemma 5.1]{Peller}
Let $\mathcal{M}$ be an invariant subspace of the operator $S$ (namely, multiplication by $z$ on $H^2(E)$). Then $\mathcal{M}$ has finite codimension if and only if $\mathcal{M}=\Theta H^2(E)$ for a Blaschke-Potapov product $\Theta$ of finite degree. Moreover,
$$\text{dim}K_{\Theta}=\text{codim}\mathcal{M}=\text{deg}~ \Theta.$$
 \end{lem}

For each $\lambda\in \mathbb{D}$ and $x\in E$, the function
$$k_{\lambda}^{\Theta}x=\frac{1}{1-\bar \lambda z}(I_{E}-\Theta(z)\Theta(\lambda)^*)x$$
belong to $K^{\infty}_{\Theta}=K_{\Theta}\cap H^{\infty}(E)$ and has the
 reproducing property
$$\langle f, k_{\lambda}^{\Theta}x\rangle=\langle f(\lambda), x\rangle \quad \text{for every}\quad f\in K_\Theta.$$
When $\Theta$ is a pure inner function, i.e., $\|\Theta(0)\|< 1$, 
$$k_{0}^{\Theta}(z)= I_{E}-\Theta(z)\Theta(0)^*,$$ which is invertible in $H^{\infty}(\oL)$.

In the case when $\Theta(\lambda)=0$, we have
$$k_{\lambda}^{\Theta}x=k_{\lambda}x.$$
The \emph{model operator} $S_{\Theta}\in\mathcal{L}(K_{\Theta})$ is defined  by the formula
\begin{equation*}\label{eq:Model Operator}
(S_{\Theta}f)(z)=P_{\Theta} (zf), \quad f\in K_{\Theta}.
 \end{equation*}
Its adjoint is given by
\[
(S_{\Theta}^{*}f)(z)=\frac{f(z)-f(0)}{z},
\]
which is the restriction of the left shift in $H^{2}(E)$ to the $S^{*}$-invariant subspace $K_{\Theta}$.

The space $\Theta H^2(E)$ is invariant under $S=M_z$, but it is not invariant under $M_\Phi$ for general $\Phi\in H^{\infty}(\oL)$, and hence $K_\Theta$ is not invariant under $M^*_\Phi$ (see \cite[Example 4.1]{RT}). Therefore, the relation $(A^{\Theta}_\Phi)^*=M^*_\Phi|_{K_\Theta}$ does not hold. However, if $\Phi\in H^{\infty}(\oL)$ commutes with $\Theta$, then $M^*_\Phi K_\Theta\subset K_\Theta$ and hence $(A^{\Theta}_\Phi)^{*}=M^*_\Phi|_{K_\Theta}$. It then follows that $(A^{\Theta}_\Phi)^{*} S^*_\Theta=S^*_\Theta (A^{\Theta}_\Phi)^*$, which gives $A^{\Theta}_\Phi S_\Theta=S_{\Theta} A^{\Theta}_\Phi$.

\medskip

The following theorem shows that the converse is also true.


\begin{thm}\cite[Chapter VI]{NF}\label{clt}
    Let $\Theta$ be an inner function, and let $A\in \mathcal{L}(K_{\Theta})$ such that $AS_{\Theta}=S_{\Theta}A$. Then there exists $\Phi\in H^\infty(\mathcal{L}(E))$ with $\Phi\Theta=\Theta\Phi$ such that $A=A_{\Phi}^{\Theta}$. Moreover 
    \begin{enumerate}
        \item For any $A=A_{\Phi}^{\Theta}$, with $\Phi\in H^{\infty}(\oL)$, we have
        $$\|A\|=dist(\Phi, \Theta H^\infty(\oL)$$ as well as 
        $$\|A\|=\inf{\{\|\Phi\|_{\infty}: A=A_{\Phi}^{\Theta},~~\Phi\in H^{\infty}(\mathcal{L}(E))}\}.$$
        
        \item There exists $\Phi\in H^\infty(\oL)$ such that $\|A\|=\|\Phi\|_{\infty}.$
    \end{enumerate}
\end{thm}
\medskip

To prove the matrix version of the Nevanlinna-Pick interpolation problem, we recall some needed notions.

Given square matrix $A=(a_{ij})$, by the notation $A\geq 0$ we mean that
$$\sum_{i,j=1}^{n}\overline{\alpha}_i\alpha_j a_{i,j}\geq 0$$
 for every $\alpha_1, \alpha_2,\dots,\alpha_n \in \mathbb{C}$. 
 If so, we say that $A$ is positive definite. 

Let $E$ be a separable complex Hilbert space. A function $K: \mathbb{D} \times \mathbb{D}\to \oL$ is called a \emph{kernel function} provided that for every choice of $n$ distinct points $\lambda_1, \lambda_2,.\dots,\lambda_n$ in $\mathbb{D}$ and any vectors $x_1,x_2,\dots,x_n$ in $E$, we have
$$\sum_{i,j=1}^n\big\langle K(\lambda_i, \lambda_j)x_j, x_i\big\rangle\geq 0.$$

\begin{prop}\cite[Proposition 2.13]{VP}\label{VP}
     Let $E$ be a reproducing kernel Hilbert space with reproducing kernel $K$. Then $K$ is a kernel function.
\end{prop}

 We shall make use of the following lemma.

\begin{lem}
    Let $\Phi\in H^\infty(\oL)$ and $\lambda\in \mathbb{D}$ such that $\Theta(\lambda)=0$. Then 
    $$A_{\Phi^*}^{\Theta}k_{\lambda}^{\Theta}x=\Phi(\lambda)^*k^{\Theta}_{\lambda}x.$$
\end{lem}

\begin{proof}
Let $f\in K_{\Theta}$ and $x\in E$. Then
\begin{align*}
    \langle A_{\Phi^*}^{\Theta}k_{\lambda}^{\Theta}x, f\rangle
   & =\langle \Phi^* k_{\lambda}^{\Theta}x, f\rangle\\
    &=\langle  k_{\lambda}^{\Theta}x, \Phi f\rangle\\
      &=\overline{\langle  \Phi f, k_{\lambda}^{\Theta}x\rangle}\\
      &=\overline{\langle  \Phi(\lambda) f(\lambda), x\rangle}\\
      &=\overline{\langle   f(\lambda), \Phi(\lambda)^* x\rangle}\\
&=\overline{\langle   f, k_{\lambda}^{\Theta}\Phi(\lambda)^* x\rangle}\\
&=\overline{\langle   f,\Phi(\lambda)^* k_{\lambda}^{\Theta} x\rangle}\\
&=\langle \Phi(\lambda)^* k_{\lambda}^{\Theta} x,  f\rangle.
\end{align*}
The required result follows.
\end{proof}

As an application of matrix-valued truncated Toeplitz operators, we discuss a matrix version of the Nevanlinna–Pick interpolation problem, which  addresses the following question.

    Let $(\lambda_i)^{d}_{i=1}\subset \mathbb{D}$ and $(W_i)_{i=1}^d\subset \oL$ with $W_j^{*}W_i=W_iW_j^{*}$. Does there exist a function $\Phi\in H^\infty(\oL)$ such that
    \begin{equation}
        \Phi(\lambda_i)=W_i, \quad \text{for}~~ 1\leq i\leq d, \quad \text{and} \quad \|\Phi\|_\infty\leq 1?\nonumber
    \end{equation}
If so, what is the characterizing condition? 

The answer is positive and the characterizing condition for the problem to have a solution is similar to that of the classical version of the theorem. 

\begin{thm}\label{NevPick}
     Let $(\lambda_i)^{d}_{i=1}\subset \mathbb{D}$ and $(W_i)_{i=1}^d\subset \oL$ such that $W_j^{*}W_i=W_iW_j^{*}$ for $1\leq i,j\leq d$. Then the following are equivalent:
     \begin{enumerate}
         \item There exist a function $\Phi\in H^\infty(\oL)$ such that \begin{equation*}
        \Phi(\lambda_i)=W_i, \quad \text{for}~~ 1\leq i\leq d, \quad \text{and} \quad \|\Phi\|_\infty\leq 1.\label{i}
    \end{equation*}
    \item The block matrix
    $$P=\Big[\frac{I-W^*_iW_j}{1-\overline{\lambda_i}\lambda_j}\Big]_{i,j=1}^{d}$$ is positive definite. \label{ii}
    \end{enumerate}
\end{thm}

\begin{proof}
Let $\Theta$ be a finite \emph{Blaschke-Potapov} product such that $\Theta(\lambda_i)=0$ for all $1\leq i\leq d$. Then the model space corresponding to $\Theta$ is given by
$$K_\Theta=\text{span}{\{k_{\lambda_i}x: \lambda_i\in \mathbb{D}}\}.$$
Define $R: K_{\Theta}\to K_{\Theta}$ by
$$Rk_{\lambda_i}x=W_{i}^*k_{\lambda_i}x.$$
Then
\begin{equation}\label{com1}
RS^*_{\Theta}k_{\lambda_i}x=R\overline{\lambda}_ik_{\lambda_i}x=W_{i}^*\overline{\lambda}_ik_{\lambda_i}x
\end{equation}
and
\begin{equation}\label{com2}
S^*_{\Theta}Rk_{\lambda_i}x=S^*_{\Theta}W^*_ik_{\lambda_i}x=\overline{\lambda}_iW^*_ik_{\lambda_i}x=W_{i}^*\overline{\lambda}_ik_{\lambda_i}x.
\end{equation}
From \eqref{com1} and \eqref{com2} we have $RS^*_{\Theta}=S^*_{\Theta}R$. By Theorem \ref{clt}, there exists $\Phi\in H^\infty(\oL)$ with $\Phi\Theta=\Theta\Phi$ such that $R=A_{\Phi^*}^{\Theta}$  and
$$Rk_{\lambda_i}x=A^\Theta_{\Phi^*}k_{\lambda_i}x=\Phi(\lambda_i)^*k_{\lambda_i}x.$$
Hence $\Phi(\lambda_i)=W_i$ for all $1\leq i\leq n$.
Furthermore, we also have
\begin{equation}\label{rphi}
\|R\|=\|\Phi\|_{\infty}.
\end{equation}
Let $x_1, x_2,...,x_d\in E$. Then
$$
X=\begin{bmatrix}
    x_1\\
    x_2\\
    \vdots\\
    x_d
\end{bmatrix}\in E\oplus E\oplus...\oplus E.
$$
Applying the block matrix $P$, we have
$$PX=\begin{bmatrix}
   \frac{I-W^*_1W_1}{1-\overline{\lambda_1}\lambda_1}x_1+ \frac{I-W^*_1W_2}{1-\overline{\lambda_1}\lambda_2}x_2+\dots+\frac{I-W^*_1W_d}{1-\overline{\lambda_1}\lambda_d}x_d \\
    \frac{I-W^*_2W_1}{1-\overline{\lambda_2}\lambda_1}x_1+ \frac{I-W^*_2W_2}{1-\overline{\lambda_2}\lambda_2}x_2+\dots+\frac{I-W^*_2W_d}{1-\overline{\lambda_2}\lambda_d}x_d\\
    \vdots\\
    \frac{I-W^*_dW_1}{1-\overline{\lambda_d}\lambda_1}x_1+ \frac{I-W^*_dW_2}{1-\overline{\lambda_d}\lambda_2}x_2+\dots+\frac{I-W^*_dW_d}{1-\overline{\lambda_d}\lambda_d}x_d
    \end{bmatrix}.$$
For $i,j=1,\dots,d$, using the assumption $W_j^*W_i=W_i^*W_j$,  we have
\ben \langle W_j^*W_i k_{\lambda_j}x_i, k_{\lambda_i} x_j\rangle&=&\langle W_i^*W_jk_{\lambda_j}x_i, k_{\lambda_i} x_j\rangle\nn\\
&=&\langle W_i^*R k_{\lambda_j} x_i,k_{\lambda_i} x_j\rangle\nn\\
&=&\langle R k_{\lambda_j}x_i, W_i k_{\lambda_i} x_j\rangle\nn\\
&=&\langle R k_{\lambda_j}x_i, R k_{\lambda_i} x_j\rangle\nn\\
&=&\langle R^*R k_{\lambda_j}x_i, k_{\lambda_i} x_j\rangle.\nn\eeqn  
We obtain the following expression of the inner product $\langle PX, X\rangle$ in terms of the  matrix with entries 
$k_{\lambda_j(\lambda_i)}$, \, $i,j=1,\dots,d$.
 
\begin{align*}
    \langle PX, X\rangle 
%
&=\sum_{i,j=1}^d\Big\langle \frac{I-W^*_jW_i}{1-\overline{\lambda}_j\lambda_i}x_i, x_j\Big\rangle\\
&=\sum_{i,j=1}^d\langle (I-W^*_jW_i)k_{\lambda_j}(\lambda_i)x_i, x_j\rangle\\
&=\sum_{i,j=1}^d\langle (I-W^*_jW_i)k_{\lambda_j}x_i, k_{\lambda_i}x_j\rangle\\
&=\sum_{i,j=1}^d\langle (I-R^*R)k_{\lambda_j}x_i, k_{\lambda_i}x_j\rangle\\
&=\sum_{i,j=1}^d\langle (I-R^*R)k_{\lambda_j}(\lambda_i)x_i, x_j\rangle.
\end{align*}
Since $k_{\lambda}x$ is a reproducing kernel function for the model space $K_\Theta$, by Proposition~\ref{VP} it follows that the matrix
$$[k_{\lambda_j}(\lambda_i)]_{i,j=1}^d$$
is positive definite.
\medskip

Noting that 
\begin{equation}\label{R}
P\geq 0 \quad \iff \quad 1-R^*R\geq 0 \quad  \iff\quad \|R\|\leq 1,
\end{equation}
from \eqref{rphi} and \eqref{R} we have $\|\Phi\|_{\infty}\leq 1$, which shows the equivalence of statements (1) and (2). 
\end{proof}

 \section{\bf{Unitary and completely non-unitary parts of MTTO's}}\label{s3}

Recall that, given a complex Hilbert space $E$,  an operator $T\in \oL$ is said to be a contraction if $\|T\|\leq 1$. An operator $T\in \oL$ is said to be \emph{completely non-unitary} (c.n.u) if there does not exist a $T$-reducing subspace $\mathcal{M}\subset E$ such that $T|_{\mathcal{M}}$ is unitary. The following result in \cite{NF} gives a decomposition of a contraction $T\in \oL$ into a direct sum of a unitary and a completely non-unitary operators on certain closed $T$-reducing subspaces of $E$.

\medskip

 \begin{thm}\cite{NF}\label{ucnu}
     Let $T\in \oL$ be a contraction.  Then $E=\h_u\oplus \h_{c.n.u}$ where $\h_u,\h_{c.n.u}\subset E$ are closed $T$-reducing subspaces of $E$ and $T$ can be decomposed as
     $$T=T|_{\hsp_u}\oplus T|_{\hsp_{c.n.u}}$$ where $T|_{\h_u}$ is unitary and $T|_{\h_{c.n.u}}$ is completely non-unitary.
 \end{thm}

\medskip

 Recall that an inner function $\Theta$ with $\|\Theta(0)\|<1$ is called pure. Throughout
this section we assume $\Theta$ to be a pure inner function.

\medskip



Let $\Phi \in L^{\infty}(\mathcal{L}(E))$ such that $\|\Phi\|_{\infty} \leq 1$. Then, $A^{\Theta}_{\Phi}$ is a contraction, since by part (1) of Theorem~\ref{clt}, 
\[
\|A^{\Theta}_{\Phi} \| \leq \|\Phi\|_{\infty} \leq 1.
\]
The unitary subspace of $A^{\Theta}_{\Phi}$ is defined as
\[
\mathcal{H} = \{f \in K_\Theta: \|(A^{\Theta}_{\Phi})^n f\| = \|f\| = \|(A^{\Theta}_{\Phi})^{*n} f\| \text{ for all }\, n \in \mathbb{N} \}.
\]
In particular, for $f \in \h$, we have
\[
\|f\| = \|A^{\Theta}_{\Phi}f\| = \|P_{\Theta} (\Phi f) \| \leq \|\Phi f \| \leq \|f\|.
\]
Hence
\[
\|P_{\Theta} (\Phi f) \| = \|\Phi f \|,
\]
which implies that $\Phi f \in K_\Theta$.  Moreover, $\|\Phi f\| = \|f \|$. In a similar fashion, using the condition $\|(A^{\Theta}_{\Phi})^{*}f \| = \|f\|$,  we obtain  $\Phi^* f \in K_\Theta$ and $\|\Phi^* f\| = \|f \|$. The same property holds for all powers $n \in \mathbb{N}$. Hence,
\[
\mathcal{H} = \{f \in K_\Theta: \Phi^{n} f, \Phi^{*n} f \in K_\Theta \text{ and } \|\Phi^n f\| = \|f\| = \|\Phi^{*n} f\| \text{ for all }\, n \in \mathbb{N} \}.
\]

\noindent Using the multiplication operator notation,
\[
\mathcal{H} = \{f \in K_\Theta: L_{\Phi}^n f, L_{\Phi}^{*n} f \in K_\Theta \text{ and } \|L_{\Phi}^n f\| = \|f\| = \|L_{\Phi}^{*n} f\| \text{ for all }\, n \in \mathbb{N} \}.
\]
 Furthermore, since $L_{\Phi}$ is a contraction, for any $n \in \mathbb{N}$, we have
 \[
\|L_{\Phi}^n f \| = \|f\| \Longleftrightarrow f \in  \ker (I - L_{\Phi}^{*n} L_{\Phi}^{n}),
\]
and
\[
\|L_{\Phi}^{*n} f \| = \|f\| \Longleftrightarrow f \in  \ker (I - L_{\Phi}^{n} L_{\Phi}^{*n}).
\]
Thus, we may express $\h$ as
\begin{equation}
\mathcal{H} = \{f \in K_\Theta: L_{\Phi}^n f, L_{\Phi}^{*n} f \in K_\Theta \text{ and } L_{\Phi}^n L_{\Phi}^{*n} f = f = L_{\Phi}^{*n} L_{\Phi}^n f \text{ for all }\, n \in \mathbb{N} \}.\label{calHform}\end{equation}
We deduce the following Lemma.
\medskip

\begin{lem}\label{unitary part}
     Let $f\in K_\Theta$ such that $L^n_{\Phi}f, L^{*n}_{\Phi}f\in K_\Theta$. Then $f\in\h$  if and only if $\Phi(\zeta)$ is unitary $\text{a.e. on } ~~\mathbb{T}.$
\end{lem}

\begin{proof} Let $f\in \hsp$. Then from the equality
$$L_{\Phi}^n L_{\Phi}^{*n} f = f = L_{\Phi}^{*n} L_{\Phi}^n f,$$ it follows that
$$\Phi(\zeta)\Phi^{*}(\zeta)f(\zeta)=f(\zeta)=f(\zeta)\Phi^{*}(\zeta)\Phi(\zeta).$$
\vskip3pt

\noindent Thus, for almost every $\zeta\in \mathbb{T}$, 
 either $f(\zeta)=0$ or $\Phi(\zeta)$ is unitary.  Since $f$ is analytic, it cannot be zero constant on a set of positive measure. Therefore, $\Phi$ must be unitary a.e. on $\mathbb{T}$. The converse is straightforward.
\end{proof}

\begin{cor}
The model space $K_{\Theta}$ has a unitary part if and only if $K_{\Theta}$ is reducing under $L_{\Phi}$ and there exists a subset $K$ of $\mathbb{T}$ of positive measure such that $\Phi(\zeta)$ is unitary on $K$.
\end{cor}
The proof follows from Lemma~\ref{unitary part}.

\begin{thm}\label{Stheta}
     Every $S_\Theta$-invariant subspace of $K_\Theta$ is of the form
     $\Theta_{1} K_{\Theta_2}$ where $\Theta_1$ and $\Theta_2$ are inner functions such that $\Theta=\Theta_1\Theta_2$.
\end{thm}

\begin{proof}
Let $\mathcal{M}\subset K_{\Theta}$ be a closed subspace such that $S_{\Theta}\mathcal{M}\subset \mathcal{M}$. Define $\mathcal{N}$ as the direct sum
$$\mathcal{N}=\mathcal{M}\oplus \Theta H^{2}(E).$$ We now show that $\mathcal{N}$ is invariant under the forward shift $S$. Take $f\in \mathcal{N}$ and let $g\in \mathcal{M}$ and $h\in H^{2}(E)$ such that
$$f=g+\Theta h\,.$$ 
Applying the forward shift $S$, we have 
\begin{eqnarray} Sf=zf=zg+z\Theta h=Sg+\Theta zh,\label{Sf1}\end{eqnarray}
where we note that the second term on the right-hand side is in $\Theta H^{2}(E)$. For the first term, we note that since $g\in \mathcal{M}\subset K_{\Theta}$, we have
$$S_{\Theta}g=P_{\Theta}(zg)=(I-P_{\Theta H^{2}(E)})(zg)=Sg-P_{\Theta H^{2}(E)}(zg),$$ which gives
$$Sg=S_{\Theta}g+P_{\Theta H^{2}(E)}(zg).$$
Since $\mathcal{M}$ is $S_{\Theta}$ invariant,  $g'=S_{\Theta}g$ is in $\mathcal{M}$ and 
\begin{eqnarray} Sg=g'+\Theta h'\label{Sg2}\end{eqnarray} for some 
 $h'\in H^{2}(E).$
Substituting (\ref{Sg2}) into (\ref{Sf1}) yields
$$Sf=
g'+\Theta(h'+zh)
\in \mathcal{M}\oplus\Theta H^{2}(E)=\mathcal{N}\,.$$
Therefore $\mathcal{N}$ is $S$-invariant. By  Beurling-Lax-Halmos theorem (see \cite{AM}), there exists an inner function $\Theta_{1}$ such that $$\mathcal{N}=\Theta_{1}H^{2}(E).$$
We also have $\Theta H^{2}(E)\subset \mathcal{N}=\Theta_{1}H^{2}(E)$. Thus there exist an analytic function $\Theta_{2}\in H^{\infty}(\mathcal{L}(E))$ such that $$\Theta=\Theta_{1}\Theta_{2}.$$ 
The function $\Theta_2$ is analytic (since $\Theta_{1}$ is a factor of $\Theta$) and can be expressed as
$$\Theta_{2}=\Theta_1^{*}\Theta.$$ Then
$$\Theta_{2}^{*}\Theta_2=(\Theta_1^{*}\Theta)^{*}(\Theta_1^{*}\Theta)=\Theta^*\Theta_1\Theta_{1}^{*}\Theta=I.$$ 
Thus $\Theta_{2}$ is an inner function. Writing $\mathcal{M}$ as
$N\ominus \Theta H^{2}(E),$ it follows that
$$\mathcal{M}=\Theta_{1}H^{2}(E)\ominus \Theta_{1}\Theta_{2}H^{2}(E)=\Theta_{1}[H^{2}(E)\ominus \Theta_{2}H^{2}(E)]=\Theta_{1}K_{\Theta_2},$$
as desired.
\end{proof}

\begin{thm}\label{inv}
    The unitary space $\h$ is $S_\Theta$-invariant if $L_{\Phi}K_{\Theta}\subseteq K_{\Theta}$.
\end{thm}

\begin{proof}  The assumption $L_{\Phi} K_{\Theta}\subseteq K_\Theta$ is equivalent to $$L_{\Phi}P_{\Theta}=P_{\Theta}L_{\Phi}.$$
    Let $f\in \h$. Then, by (\ref{calHform}), we obtain 
$$L_\Phi S_\Theta f=L_\Phi P_\Theta (zf)=P_\Theta L_\Phi L_z f=P_\Theta L_z L_\Phi f=S_{\Theta}L_\Phi f\in K_\Theta.$$
Repeating this process, we see that $S_\Theta L^n_\Phi=L^{n}_{\Phi}S_\Theta$ for all $n\in \mathbb{N}$. Similarly, we have $S_\Theta L^{*n}_\Phi=L^{*n}_{\Phi}S_\Theta$.
Hence,
    $$L^n_{\Phi}S_\Theta f=S_\Theta L^n_\Phi f\in K_\Theta$$ and
    $$L^{*n}_{\Phi}S_\Theta f=S_\Theta L^{*n}_\Phi f\in K_\Theta.$$
     Therefore, we have
    $L_{\Phi}^n S_\Theta f, L_{\Phi}^{*n} S_\Theta f\in K_\Theta$ for all $f\in \h$ and for all $n\in \mathbb{N}$.

    Furthermore, we have
    $$L^n_{\Phi}L^{*n}_{\Phi}S_\Theta f=S_\Theta L^n_{\Phi}L^{*n}_{\Phi} f=S_\Theta f$$
    and
    $$L^{*n}_{\Phi}L^{n}_{\Phi}S_\Theta f=S_\Theta L^{*n}_{\Phi}L^{n}_{\Phi} f=S_\Theta f.$$
This shows that $S_\Theta f\in \h$, hence $\h$ is $S_\Theta$-invariant.\end{proof}

We now provide a counterexample to show that the converse of Theorem \ref{inv} does not hold.
\begin{ex}
Let $E=\mathbb{C}^{2}$ and consider
$$
\Theta(z)=
\begin{pmatrix}
z&0\\
0&z
\end{pmatrix}
=zI_{E}.
$$
The model space corresponding to $\Theta$ is
$$K_{\Theta}=\left\{\begin{pmatrix}
    a\\
    b
\end{pmatrix}:  a, b\in \mathbb{C}\right\}.$$
Let
$$
\Phi(z)=
\begin{pmatrix}
1&0\\
0&z
\end{pmatrix}.
$$
Then $\Phi\in L^{\infty}(\oL)$ and $\|\Phi\|_{\infty}=1$. Observe that $A_{\Phi}^{\Theta}$ acts on $K_{\Theta}$ as a matrix. Indeed for $f=\begin{pmatrix}
    a\\
    b
\end{pmatrix}\in K_{\Theta},$
$$A_{\Phi}^{\Theta}f=P_{\Theta}(\Phi f)=P_{\Theta}
\begin{pmatrix}
a\\
bz
\end{pmatrix}
=\begin{pmatrix}
a\\
0
\end{pmatrix}=\begin{pmatrix}
1&0\\
0&0
\end{pmatrix}f.$$
 Thus, $A_{\Phi}^{\Theta}$ is a projection. Its unitary space is
$$\hsp=\text{ker}(I-(A_{\Phi}^{\Theta})^{*}A_{\Phi}^{\Theta})\cap \text{ker}(I-A_{\Phi}^{\Theta}(A_{\Phi}^{\Theta})^{*}).$$
Since  $A_{\Phi}^{\Theta}(A_{\Phi}^{\Theta})^{*}=A_{\Phi}^{\Theta}=(A_{\Phi}^{\Theta})^{*}A_{\Phi}^{\Theta}$, it follows that 
$$\hsp=\text{ker}(I-A_{\Phi}^{\Theta})=\text{Rang} A_{\Phi}^{\Theta}=\text{span}\left\{\begin{pmatrix}
    1\\
    0
\end{pmatrix}
\right\}.$$

Observe that the model operator $S_{\Theta}$ is trivial on $K_{\Theta}$. 
Indeed, for $f=\begin{pmatrix}
    a\\
    b
\end{pmatrix}\in K_{\Theta}$, 
$$S_{\Theta}f=P_{\Theta}(zf)=P_{\Theta}\begin{pmatrix}
    az\\
    bz
\end{pmatrix}=0.$$ 
Therefore,
$\hsp$ is trivially $S_{\Theta}$-invariant. However, the inclusion $L_{\Phi}K_{\Theta}\subset K_{\Theta}$ does not hold. To see this, consider $f=\begin{pmatrix}
    1\\
    1
\end{pmatrix}\in K_{\Theta}$. Then
$$L_{\Phi}f=\Phi f=\begin{pmatrix}
    1&0\\
    0&z
\end{pmatrix}\begin{pmatrix}
    1\\
    1
\end{pmatrix}=\begin{pmatrix}
    1\\
    z
\end{pmatrix}\notin K_{\Theta}.$$ We conclude that
$$L_{\Phi}K_{\Theta}\not\subset K_{\Theta}.$$
\end{ex}
\medskip

The model space corresponding to the inner function $\Theta=\Theta_1 \Theta_2$ can be decomposed as $$K_\Theta=K_{\Theta_1}\oplus \Theta_1 K_{\Theta_2}.$$ By Theorem \ref{Stheta}, $\Theta_1 K_{\Theta_2}$ is invariant under $S_{\Theta}$. Its orthogonal complement $K_{\Theta_1}=[\Theta_1 K_{\Theta_2}]^{\perp}$ is invariant under  $S^*_\Theta$. 

    \medskip
By \cite{RT}, every matrix-valued truncated Toeplitz operator has a symbol in a certain class. Denote by $M_{\Theta}$ the orthogonal complement of $\Theta H^2(\oL)$ in $H^2(\oL)$ endowed with the Hilbert–Schmidt norm. 
%
%
%
Moreover, $$L^2(\oL)=M_{\Theta}+\Theta H^2(\oL)+M_{\Theta}^{*}+[\Theta H^2(\oL)]^{*}.$$

\begin{thm}\label{Phi constant}
The model space $K_{\Theta}$ is invariant with respect to $L_{\Phi}$ if and only if $\Phi(\zeta)$ is constant.
\end{thm}

\begin{proof}
Assume $K_{\Theta}$ is invariant under $L_{\Phi}$. Then
$$A_{\Phi}^{\Theta}f=P_{\Theta}(\Phi f)=\Phi f=L_{\Phi}f.$$
Thus, $L_{\Phi}=A_{\Phi}^{\Theta}$.

Now we show that $\Phi$ is bounded and analytic. Take $f=k_{0}^{\Theta}x\in K_{\Theta}$ for any $x\in E$, we have 
$$\Phi(z)k_{0}^{\Theta}x=\Phi (I-\Theta(z)\Theta(0)^{*})x\in H^2(E).$$ Since $\Theta$ is pure, i.e., $\|\Theta(0)\|<1$, we have
$$(I-\Theta(z)\Theta(0)^{*})^{-1}\in H^{\infty}(\oL).$$ Therefore
$$\Phi x=\Phi (I-\Theta(z)\Theta(0)^{*})(I-\Theta(z)\Theta(0)^{*})^{-1}x\in H^2(\oL).$$ It follows that 
$$\Phi\in H^2(\oL)\cap L^{\infty}(\oL)=H^{\infty}(\oL).$$

Since $L_{\Phi}$ commutes with multiplication operator $L_{z}$, we have 
$$S_{\Theta}L_{\Phi}=L_{\Phi}S_{\Theta}.$$
By Theorem \ref{clt}, there exists $\Psi\in H^{\infty}(\mathcal{L}(E))$ such that $\Psi\Theta=\Theta \Psi$ and 
$$L_{\Phi}=A_{\Psi}^{\Theta}.$$
Hence,
$$A_{\Psi}^{\Theta}S_{\Theta}=S_{\Theta}A_{\Psi}^{\Theta}.$$
But since $L_{\Phi}=A_{\Phi}^{\Theta}$ on $K_{\Theta}$, we have 
\begin{equation}\label{MTTOEQ}
A_{\Phi}^{\Theta}=A_{\Psi}^{\Theta} 
\end{equation}
on $K_{\Theta}$.
Using the reproducing kernel property and invariance of $K_{\Theta}$ under $L_{\Phi}$, for each $\lambda\in \mathbb{D}$, $x\in E$, and $f\in K_\Theta$, we have 
\begin{align}
\langle (A_{\Phi}^{\Theta})^{*}k_{\lambda}^{\Theta}x, f\rangle
&=\langle k_{\lambda}^{\Theta}x, \Phi f\rangle\nn\\
&=\overline{\langle  \Phi f, k_{\lambda}^{\Theta}x\rangle}\nn\\
&=\overline{\langle  \Phi(\lambda) f(\lambda), x\rangle}\nn\\
&=\overline{\langle   f(\lambda), \Phi(\lambda)^{*} x\rangle}\nn\\
&=\overline{\langle   f, k_{\lambda}^{\Theta}\Phi(\lambda)^{*} x\rangle}\nn\\
&=\langle k_{\lambda}^{\Theta}\Phi(\lambda)^{*} x, f\rangle.\label{F1}
\end{align}
Therefore
$$(A_{\Phi}^{\Theta})^{*}k_{\lambda}^{\Theta}x= k_{\lambda}^{\Theta}\Phi(\lambda)^{*} x.$$
Since $\Phi,\Psi\in H^{\infty}(\oL)$, and  $$A_{\Phi}^{\Theta}=A_{\Psi}^{\Theta},$$ then by \cite[Corollary 6.4]{RT} there exist $k_{0}^{\Theta}\in M_{\Theta}$ and $X\in \oL$ such that 
$$\Psi=\Phi+k_{0}^{\Theta}X.$$
Observe that for each $\lambda\in \mathbb{D}$, $x\in E$, and $f\in K_{\Theta}$, 
\ben \langle k_{\lambda}^{\Theta}x, k_{0}^{\Theta}X f \rangle&=&\overline{\langle k_{0}^{\Theta}X f, k_{\lambda}^{\Theta}x\rangle}\nn\\
&=&\overline{\langle k_{0}^{\Theta}(\lambda)X f(\lambda), x\rangle}\nn\\
&=&\overline{\langle  f(\lambda), X^{*}(k_{0}^{\Theta}(\lambda))^{*}x\rangle}\nn\\
&=&\overline{\langle  f, k_{\lambda}^{\Theta}X^{*}(k_{0}^{\Theta}(\lambda))^{*}x\rangle}\nn\\
&=&\langle  k_{\lambda}^{\Theta}(k_{0}^{\Theta}(\lambda)X)^{*}x, f\rangle.\label{F2}\eeqn
From (\ref{F1}) and (\ref{F2}), we obtain
\begin{align*}
\langle (A_{\Psi}^{\Theta})^{*}k_{\lambda}^{\Theta}x, f \rangle
&=\langle k_{\lambda}^{\Theta}x, \Psi f \rangle\\
&=\langle k_{\lambda}^{\Theta}x, \Phi f \rangle+\langle k_{\lambda}^{\Theta}x, k_{0}^{\Theta}X f \rangle\\
&=\langle k_{\lambda}^{\Theta}\Phi(\lambda)^{*}x, f    \rangle+\langle  k_{\lambda}^{\Theta}(k_{0}^{\Theta}(\lambda)X)^{*}x, f\rangle\\
&=\langle k_{\lambda}^{\Theta}(\Phi(\lambda)^{*}x+k_{\lambda}^{\Theta}(k_{0}^{\Theta}(\lambda)X)^{*}x, f \rangle \\
&=\langle k_{\lambda}^{\Theta}(\Phi(\lambda)+k_{0}^{\Theta}(\lambda)X)^{*}x, f \rangle \\
&=\langle k_{\lambda}^{\Theta}\Psi(\lambda)^{*}x, f \rangle. 
\end{align*}
Hence
$$(A_{\Psi}^{\Theta})^{*}k_{\lambda}^{\Theta}x=k_{\lambda}^{\Theta}\Psi(\lambda)^{*}x.$$ Since $$A_{\Phi}^{\Theta}=A_{\Psi}^{\Theta}$$ and the kernel map $x\to k_{\lambda}^{\Theta}x$ is injective, we get 
$$\Phi(\lambda)^{*}=\Psi(\lambda)^{*}, ~~\text{for all}~~\lambda\in \mathbb{D}.$$ Hence, $\Phi=\Psi$ as functions.

Next observe that, by the invariance of $ K_{\Theta}$ under $L_{\Phi}$,  for all $f\in K_{\Theta}$ we have $\Phi f\in K_{\Theta}$, which means that
$$0=\langle \Phi f, \Theta h \rangle=\langle f, \Phi^{*}\Theta h\rangle $$ for all $h\in H^{2}(E)$. It follows that $\Phi^{*}\Theta h\in \Theta H^{2}(E)$. By \cite[Lemma 2.2]{RT}, there exists $\Phi_{1}\in H^{\infty}(\mathcal{L}(E))$ such that $$\Phi^{*}\Theta=\Theta \Phi_{1}.$$
From this, using the identity $\Phi\Theta=\Theta\Phi$, we have
\begin{eqnarray} \Phi_{1}&=&\Theta^{*}\Phi^{*}\Theta\nn\\
&=&(\Phi\Theta)^{*}\Theta\nn\\
&=&(\Theta\Phi)^{*}\Theta\nn\\
&=&\Phi^{*}\Theta^{*}\Theta\nn\\
&=& \Phi^{*}\,.\label{Phi1=Phi*}\end{eqnarray}

\noindent Since $\Phi_{1}\in H^{\infty}(\mathcal{L}(E))$, $\Phi_{1}$ is analytic. Thus, from (\ref{Phi1=Phi*}) it follows that
$$\Phi^{*}\in H^{\infty}(\mathcal{L}(E)).$$
On the other hand, as observed above, $\Phi\in H^{\infty}(\mathcal{L}(E))$ is also analytic.  Hence, the only possibility is that $\Phi(\zeta)$ is constant.

The proof of the converse is straightforward. Hence it is left to the reader.
\end{proof}



\medskip

\medskip

\medskip
\section{\bf{Invertibility and solvability of matrix-valued asymmetric truncated Toeplitz operators}} \label{s4}
\subsection{Invertibility}
When the symbol $\varphi$ is the constant $1$, in \cite{HNP}, using Devinatz--Windom's theorem, 
Hru$\check{s}\check{c}$ev,  Nikol'skii,  Pavlov obtain the following link with the invertibility of a classical Toeplitz operator associated to a unimodular symbol. 

\begin{thm}\cite[Theorem 4]{HNP}\label{1}
Let $\theta_1,\theta_2$ be inner functions. The following statements are equivalent.
\begin{enumerate}
\item The operator $A_{1}^{\theta_{1},\theta_{2}}=P_{\theta_{2}}:K_{\theta_{1}}\to K_{\theta_{2}}$ is invertible.
\item The Toeplitz operator $T_{\overline{\theta}_{2}\theta_{1}}:H^{2}\to H^{2}$ is invertible.
\item $\mbox{dist}(\overline{\theta_1}\theta_2,H^\infty)<1$ and $\mbox{dist}(\overline{\theta_2}\theta_1,H^\infty)<1$.
\end{enumerate}
\end{thm}

\medskip
Recently in \cite[Theorem 5.6]{part3}, C\^{a}mara and Partington provide 
 an interesting connection  between the invertibility of an asymmetric truncated Toeplitz operator and of  the Toeplitz operator with matrix symbol. Their theorem was proved for the domain $\mathbb{R}$, after establishing several lemmas in the scalar case. Below, we present a new 
 simplified (algebraic) proof of the matrix version for the unit circle $\mathbb{T}$ of their theorem. 
 We wish to emphasize that the result of C\^{a}mara and Partington is, however, more general for the scalar case.

\medskip

\begin{thm}\label{G} Let $\Theta_1$ and $\Theta_2$ be inner functions and 
 $\Phi\in L^{\infty}(\oL)$. Then $A_{\Phi}^{\Theta_{1},\Theta_{2}}$ is invertible if and only if the Toeplitz operator $T_{G}: H^2(E)\oplus H^2(E)\to H^2(E)\oplus H^2(E)$ is invertible,
where $G=\begin{pmatrix}
\Theta_{2}& \Phi\\
0&\Theta_{1}^{*}
\end{pmatrix}.
$
\end{thm}
\begin{proof}
Consider the equation generated by the operator $A_{\Phi}^{\Theta_{1},\Theta_{2}}\in \mathcal{T}(\Theta_{1},\Theta_{2})$, i.e.,
\begin{equation}\label{asymm}
A_{\Phi}^{\Theta_{1},\Theta_{2}}f_{\Theta_{1}}=P_{\Theta_{2}}(\Phi f_{\Theta_{1}})=g_{\Theta_{2}},
\end{equation}
where $f_{\Theta_{1}}\in K_{\Theta_{1}}$ and $g_{\Theta_{2}}\in K_{\Theta_{2}}$.
\medskip

The equation generated by Toeplitz operator $T_{G}$ with block matrix symbol is given by
 \begin{equation}\label{TG}
P_{+}\begin{pmatrix}
\Theta_{2}& \Phi\\
0&\Theta_{1}^{*}
\end{pmatrix}
\begin{pmatrix}
X_{1}\\
X_{2}
\end{pmatrix}
=\begin{pmatrix}
f_{1}\\
f_{2}
\end{pmatrix},
\end{equation}
where $X_{1}, X_{2}, f_{1}, f_{2}\in H^{2}(E)$.

Suppose $A_{\Phi}^{\Theta_{1},\Theta_{2}}:K_{\Theta_{1}}\to K_{\Theta_{2}}$ is invertible. We wish to prove that \eqref{TG} has unique solution. Performing block multiplication, this equation can be written as the system of two equations
%
 \begin{equation}\label{first}
 P_{+}(\Theta_{2}X_{1})+P_{+}(\Phi X_{2})=f_{1}
 \end{equation}
 \begin{equation}\label{second}
 P_{+}(\Theta^{*}_{1}X_{2})=f_{2}.
 \end{equation}
 The general solution of \eqref{second} has the  form
 $$X_{2}=\psi+\Theta_{1}f_{2},$$
 where $\psi$ is an arbitrary function in 
  $K_{\Theta_{1}}$. Substituting this into 
  \eqref{first} yields 
 \begin{equation}\label{third}
 \Theta_{2}X_{1}+P_{+}(\Phi\psi)+P_{+}(\Phi\Theta_{1}f_{2})=f_{1}.
 \end{equation}
 Applying the projection operator $P_{\Theta_{2}}$ to both sides of \eqref{third}, we obtain
  \begin{equation}
 P_{\Theta_{2}}P_{+}(\Phi\psi)+P_{\Theta_{2}}P_{+}(\Phi\Theta_{1}f_{2})=P_{\Theta_{2}}(f_{1}).\nn
 \end{equation}
Equivalently 
 \begin{equation}\label{5th}
 A_{\Phi}^{\Theta_{1},\Theta_{2}} \psi=P_{\Theta_{2}}(f_{1}-P_{+}(\Phi\Theta_{1}f_{2})).
 \end{equation}
 By the invertibility of $A_{\Phi}^{\Theta_{1},\Theta_{2}}$,  
  equation \eqref{5th} has unique solution $\psi\in K_{\Theta_{1}}$, which in turn determines $X_2$. Since, $T_{\Theta_{2}}$ is isometric with range equal to $\Theta_2 H^2(E)$, by \eqref{third}, we have
  $$
 T_{\Theta_{2}}X_{1}=f_1-P_{+}(\Phi\psi)-P_{+}(\Phi\Theta_{1}f_{2})=h\in \Theta_2 H^2(E).
 $$
The equation
 \begin{equation}\label{X_1h}
     T_{\Theta_2}X_1=h
 \end{equation}
has at most one solution. 
Since $T_{\Theta_2}$  is injective and $h$ is in its range, there is unique $X_1\in H^2(E)$ satisfying \eqref{X_1h}. Explicitly, $$X_1=\Theta_2^{*} h\in H^2(E)$$
satisfying
 $$T_{\Theta_2}(X_1)=T_{\Theta_2}(\Theta_2^{*}h)=h.$$ 
 
Conversely, assume \eqref{TG} has unique solution for arbitrary $f_{1}, f_{2}\in H^{2}(E)$. We may express  \eqref{asymm} in the form
 \begin{equation}\label{atto}
 P_{+}(\Phi f_{\Theta_{1}})+\Theta_{2}X_{1}=g_{\Theta_{2}},
 \end{equation}
 where $X_{1}$ is some unknown function in 
  $H^{2}(E)$. Along with \eqref{atto}, we consider an equation of the kind
 \begin{equation}\label{zero}
 P_{+}(\Theta_{1}^{*}f_{\Theta_{1}})=0.
 \end{equation}
 Rewriting \eqref{atto} and \eqref{zero} of the form:
 \begin{equation}\label{last}
P_{+}\begin{pmatrix}
\Theta_{2}& \Phi\\
0&\Theta_{1}^{*}
\end{pmatrix}
\begin{pmatrix}
X_{1}\\
f_{\Theta_{1}}
\end{pmatrix}
=\begin{pmatrix}
g_{\Theta_{2}}\\
0
\end{pmatrix},
\end{equation}
such that \eqref{last} has unique solution, we deduce that the operator $A_{\Phi}^{\Theta_{1},\Theta_{2}}$ is invertible. Hence, the required result follows.
\end{proof}
\subsection{Solvability}
     An operator $T\in \mathcal{L}(E)$ is said to be normally solvable if the range of $T$ is closed in $E$.
\medskip


    \begin{rem}\label{Blockrange}
    If $S, T\in \oL$ have closed range, then the block-triangular operator 
$$B=\begin{pmatrix}
    S&A\\
    0&T
\end{pmatrix}$$ also has closed range, provided that $A$ is bounded.
Moreover, range of 
$B$ can be described as
$$\text{Ran}~ (B)=\text{Ran}~(S)\oplus \text{Ran}~(T).$$
\end{rem}

We now use Remark~\ref{Blockrange} to prove the following result.

\begin{prop} \label{prTG} Let $\Theta_1$ and $\Theta_2$ be inner functions and 
 $\Phi\in L^{\infty}(\oL)$. Then the operator $T_{G}$ is normally solvable.
\end{prop}
\begin{proof}
Observe that $T_{G}$ is the block-triangular Toeplitz operator
$$T_{G}=
\begin{pmatrix}
    T_{\Theta_2}& T_{\Phi}\\
    0& T_{\Theta^{*}_{1}}
\end{pmatrix}
$$
whose block diagonal entries are normally solvable. Indeed, since the symbol $\Theta_2$ is inner, the Toeplitz operator $T_{\Theta_2}$ is an isometry on $H^2(E)$, and hence has closed range. The operator $T_{\Theta^{*}_{1}}$ is a co-isometry, hence it has closed range also. 
By the boundedness of $T_\Phi$ and Remark \ref{Blockrange} it follows that $T_{G}$ has closed range and hence it is normally solvable.
\end{proof}
In the following theorem we discuss the normal solvability of the asymmetric truncated Toeplitz operators.


\begin{thm} \label{ATTO} Let $\Theta_1$ and $\Theta_2$ be inner functions and 
 $\Phi\in L^\infty(\oL)$.  Then the operator $A_{\Phi}^{\Theta_1, \Theta_2}$ is normally solvable.
\end{thm}

\begin{proof}
To prove that $A_{\Phi}^{\Theta_1, \Theta_2}$ is normally solvable, we need to show that if $(k_n)\subset K_{\Theta_1}$ is a sequence of functions such that
    $$A_{\Phi}^{\Theta_1, \Theta_2}k_n\to h,$$
 then $h\in \text{Rang}(A_{\Phi}^{\Theta_1, \Theta_2})$.

Let $(k_n)$ be such a sequence. Then $(0, k_n)$ is in $H^2(E)\oplus H^2(E)$ and $T_{G}$ acts on 
$(0, k_n)$ as follows:
    $$T_{G}(0, k_n)=(P_{+}(\Phi k_n), P_{+}(\Theta^{*}_1 k_n)).$$    
Since $k_{n}\in K_{\Theta_1}$, we have
    $$P_{+}(\Theta^{*}_1 k_{n})=0.$$
    Hence,
    $$T_{G}(0, k_n)\to (h, 0)~~\text{in} ~~H^2(E)\oplus H^2(E).$$
    Since by Proposition \ref{prTG},  $T_{G}$ has a close range, the limit $(h, 0)$ must belong to $\text{Rang}(T_{G})$. So there exist $(f, g)\in H^2(E)\oplus H^2(E)$ such that $$T_{G}(f, g)=(h, 0).$$ Hence, we have
    $$(P_{+}(\Theta_2 f+\Phi g), P_{+}(\Theta^{*}_1 g))=(h, 0).$$
    From the equality of the second components, $P_{+}(\Theta^{*}_1 g)=0,$ it follows that  $g\in K_{\Theta_1}$. From the first components, we have
    $$P_{+}(\Theta_2 f+\Phi g)=h.$$
    Taking the projection on $K_{\Theta_2}$ we have
    $$P_{\Theta_2}(\Phi g)=h,$$ which is exactly
    $$A_{\Phi}^{\Theta_1, \Theta_2}(g)=h.$$ Hence,
    $\text{Rang}(A_{\Phi}^{\Theta_1, \Theta_2})$ is closed.
\end{proof}

\end{document}